\documentclass[reqno,11pt]{amsart}

\newtheorem{theorem}{Theorem}[section]
\newtheorem{lemma}[theorem]{Lemma}
\newtheorem{proposition}[theorem]{Proposition}

\theoremstyle{definition}

\theoremstyle{remark}

\newcommand{\mysection}[1]{\section{#1}
\setcounter{equation}{0}}

\newcommand{\bR}{\mathbb R}

\newcommand\cL{\mathcal{L}}

\def\bH{\mathbb{H}}
\def\cH{\mathcal{H}}

\newcommand{\esssup}{\operatorname*{ess\,sup}}
\renewcommand{\epsilon}{\varepsilon}

\begin{document}
\title[Partial regularity for 4D NSE]{Partial Regularity of solutions to the Four-Dimensional Navier-Stokes Equations}
\author[H. Dong]{Hongjie Dong}
\address[H. Dong]{Division of Applied Mathematics, Brown University,
182 George Street, Providence, RI 02912, USA}
\email{Hongjie\_Dong@brown.edu}

\author[X. Gu]{Xumin Gu}
\address[X. Gu]{School of Mathematical Sciences, Fudan University,
Shanghai 200433, People's Republic of China}
\email{xumin\_gu@brown.edu;11110180030@fudan.edu.cn}

\begin{abstract}
In this paper, we consider  suitable weak solutions of incompressible Navier--Stokes equations in four spatial dimensions. We prove that the two-dimensional time-space Hausdorff measure of the set of singular points is equal to zero.
\end{abstract}

\date{\today}
\subjclass{35Q30, 76D03, 76D05}

\keywords{Navier-Stokes equations, partial regularity, Hausdorff's
dimension}

\maketitle

\mysection{Introduction}
In this paper we consider the incompressible Navier--Stokes equations in four spatial dimensions with unit viscosity and an external force:
\begin{eqnarray}
\label{ns}
u_t +u\cdot \nabla u -\Delta u + \nabla p &=& f\\
\nabla \cdot u &=& 0
\label{inc}
\end{eqnarray}
in a bounded cylindrical domain $Q_T\equiv\Omega\times(0,T)$, where $\Omega \subset \mathbb{R}^4$. We are interested in the partial regularity of suitable weak solutions $(u,p)$ to \eqref{ns}-\eqref{inc}.

We say that a pair of functions $(u,p)$ is a suitable weak solution to \eqref{ns}-\eqref{inc} in $Q_T$ if $u \in L_{\infty}(0,T;L_2(\Omega;\mathbb{R}^4))\cap L_2(0,T;W_2^1(\Omega;\mathbb{R}^4))$ and $p \in L_{3/2}(Q_T)$ satisfy \eqref{ns}-\eqref{inc} in the weak sense and additionally the generalized local energy inequality holds for all non-negative functions
$\psi \in C_0^{\infty}(Q_T)$:
\begin{multline}
\esssup_{0<s\leq t}\int_{\Omega}|u(x,s)|^2\psi(x,s) \,dx +2\int_{Q_t}|\nabla u|^2\psi \,dx\,ds\\
\leq \int_{Q_t}|u|^2(\psi_t+\Delta \psi)+(|u|^2+2p)u\cdot \nabla \psi +f\cdot u \psi \,dx\, ds.
\label{energy}
\end{multline}

We will prove that for any suitable weak solution $(u,p)$, the two dimensional space-time Hausdorff measure of the set of singular points is equal to zero.

The problem of the global regularity of solutions to the Navier--Stokes equations in three and higher space dimensions is a fundamental question in fluid dynamics and is still widely open. Meanwhile, many authors have studied the partial regularity of solutions. In three dimensional case, Scheffer established various results for weak solutions in \cite{Scheffer_76, Scheffer_77}. In a celebrated paper \cite{CKN_82}, Caffarelli, Kohn, and Nirenberg firstly introduced the notion of suitable weak solutions, which satisfy a local energy inequality. They proved that for any suitable weak solution, there is an open subset where the velocity field $u$ is regular and the 1D Hausdorff measure of the complement of this subset is equal to zero. In \cite{Lin_98}, Lin gave a more direct and simplified proof of Caffarelli, Kohn and Nirenberg's result with zero external force. Ladyzhenskaya and Seregin gave a detailed account in \cite{Lady_99} later. We also refer the reader to Tian and Xin \cite{Tian_99}, He \cite{He_04}, Seregin \cite{Seregin_06}, Gustafson, Kang, and Tsai \cite{Gus_07}, Vasseur \cite{Vasseur_07}, Kukavica \cite{Kukavica_08}, and the references therein for extended results.

For the four or higher dimensional Navier--Stokes equations, the problem is more super-critical. In \cite{Scheffer_78}, Scheffer showed that there exists a weak solution $u$ in $\mathbb{R}^4 \times \mathbb{R}^{+}$, which may not satisfy the local energy inequality, such that $u$ is continuous except for a set whose 3D Hausdorff measure is finite. In \cite{Dong_07}, Dong and Du showed that for any local-in-time smooth solution to 4D Navier--Stokes equations, the 2D Hausdorff measure of the set of singular points at the first potential blow-up time is equal to zero. Moreover, for stationary high dimensional Navier--Stokes equations, Struwe \cite{struwe_88} proved that suitable weak solutions are
regular outside a singular set of zero 1D Hausdorff measure in $\mathbb{R}^5$, and Kang \cite{kang_06} improved Struwe's result up to the boundary for a smooth domain $\Omega \subset \mathbb{R}^5$. Recently, Dong and Strain \cite{Dong_12} studied the partial regularity for suitable weak solutions of 6D stationary Navier--Stokes equations, and proved that solutions are regular outside a singular set of zero 2D Hausdorff measure. Based on Campanato's approach, the main idea in \cite{Dong_12} is to first establish a weak decay estimate of certain scaling invariant quantities, and then successively improve this decay estimate by a bootstrap argument and the
elliptic regularity theory.

Because time corresponds two space dimensions, in some sense the 4D non-stationary Navier--Stokes equations is similar to 6D stationary problem. Given the result in \cite{Dong_12}, it is natural to ask whether Caffarelli--Kohn--Nirenberg's theorem can be extended the 4D non-stationary case. Here the main difficulty stems from the fact that certain compactness arguments appeared, for instance, in the original paper \cite{CKN_82} as well as \cite{Lin_98, Lady_99} break down in the 4D case. We note that the results obtained in the \cite{Dong_07} cannot be considered as a genuine extension of the theorem, as the set of singular points is only estimated at the first blow-up time for local smooth solutions. The objective of this paper is to give a complete answer to this question.

We state our main results, where we use some notation introduced at the beginning of the next section.
\begin{theorem}
\label{mainthm}
Let $\Omega$ be a open set in $\mathbb{R}^4$, $f \in L_{6,\text{loc}}(Q_T)$.  Let $(u,p)$ be a suitable weak solution of \eqref{ns}-\eqref{inc} in $Q_T$. There is a positive number $\epsilon_0$ satisfying the following property. Assume that for point a $z_0 \in Q_T$, the inequality
\begin{equation*}
\limsup_{r\searrow 0} E(r) \leq \epsilon_0
\end{equation*}
holds. Then $z_0$ is a regular point.
\end{theorem}

\begin{theorem}
\label{th2}
Let $\Omega$ be a open set in $\mathbb{R}^4$ $f \in L_{6,\text{loc}}(Q_T)$. Let $(u,p)$ be a suitable weak solution of \eqref{ns}-\eqref{inc} in $Q_T$.
There is a positive number $\epsilon_0$ satisfying the following property. Assume that for a point $z_0 \in Q_T$ and for some $\rho_0 >0$ such that $Q(z_0,\rho_0) \subset Q_T$ and
\begin{equation*}
C(\rho_0)+D(\rho_0)+F(\rho_0)+G(\rho_0) \leq \epsilon_0.
\end{equation*}
Then $z_0$ is a regular point.
\end{theorem}

\begin{theorem}
\label{th3}
Let $\Omega$ be a open set in $\mathbb{R}^4$, $f \in L_{6,\text{loc}}(Q_T)$. Let $(u,p)$ be a suitable weak solution of \eqref{ns}-\eqref{inc} in $Q_T$. Then the 2D Hausdorff measure of the set of singular points in $Q_T$ is equal to zero.
\end{theorem}

Compared to \cite{Dong_07}, as mentioned above Dong and Du concerned with local-in-time smooth solution to 4D Navier--Stokes equations with zero external force. Due to the lack of compactness, they used Schoen's trick in their proof. In our paper, we shall consider suitable weak solutions, and thus Schoen's trick is no longer applicable. Our proofs exploit the aforementioned idea in \cite{Dong_12} and use Campanato's approach. There are two main differences between our problem with the one in \cite{Dong_12}. The first one is that we do not have the same end-point Sobolev embedding inequality which was used in \cite{Dong_12}. To this end, we introduce an additional scale-invariant quantity $F$, which is a mixed space-time norm of the pressure $p$, and use an interpolation inequality.  As a consequence, we cannot archive the same optimal decay rate as in \cite{Dong_12}. Nevertheless, it turns out that the decay rate, although not optimal, still suffices for our purpose in the subsequent step. The other difference is that, as our problem is time-dependent, we cannot use the elliptic regularity theory to improve the decay rate in the final step as in \cite{Dong_12}. Naturally, we appeal to the parabolic regularity theory instead as well as a Poincar\'e type inequality for solutions to divergence form parabolic equations.

We remark that by using the same method we can get an alternative proof of  Caffarelli--Kohn--Nirenberg's theorem for the 3D Navier--Stokes equations without using any compactness argument. It remains an interesting open problem whether a similar result can be obtained for five or higher dimensional non-stationary Navier--Stokes equations. It seems to
us that four is the highest dimension to which our approach (or any
existing approach) applies. In fact, by the imbedding theorem, we have
$$
L_{\infty}((0,T);L_2(\Omega))\cap L_2((0,T);W_2^1(\Omega))\hookrightarrow L_{2(d+2)/d}((0,T)\times \Omega),
$$
which implies nonlinear term in the energy inequality cannot
be controlled by the energy norm alone when $d \ge 5$.

We organize this paper as follows: In Section \ref{s1}, we introduce the notions of some scaling invariant quantities and essential settings which would be used throughout the paper. In Section \ref{s2}, we prove our results in three steps. In the first step, we give some estimates of the scaling invariant quantities, which essentially follows the argument in \cite{Dong_07}. In the second step, we establish a weak decay estimate of certain scaling invariant quantities based on the estimate we proved in the first step by using an iteration metho. In the last step, we improve the decay estimate by a bootstrap argument, and apply parabolic regularity to get a good estimate of the $L_{3/2}$-mean oscillations of $u$, which yields the H\"{o}lder continuity of $u$ according to Campanato's characterization of H\"older continuous functions.

\mysection{Notation and Settings}\label{s1}

In this section, we will introduce the notation which would be used throughout the article. Let $\Omega$ be a domain in some finite-dimensional space. Denote $L_{p}(\Omega;\mathbb{R}^n)$ and $W_p^k(\Omega;\mathbb{R}^n)$ to be the usual Lebesgue and Sobolev spaces of functions from $\Omega$ into $\mathbb{R}^n$. Let $p\in (1,\infty)$ and $-\infty\le S<T\le \infty$. We denote $\cH_p^1$ to be the solution spaces for divergence form parabolic equations. Precisely,
$$
\cH^1_p(\Omega\times (S,T) )=
\{u: u,Du \in L_p(\Omega\times (S,T) ),\,u_t \in \bH^{-1}_p(\Omega\times (S,T) \},
$$
where $\bH^{-1}_p(\Omega\times (S,T) )$ is the space consisting of all generalized functions $v$ satisfying
$$
\inf \big\{\|f\|_{L_p(\Omega\times (S,T)) }+\|g\|_{L_p(\Omega\times (S,T)) }\,|\,v=\nabla\cdot g+f\big\}<\infty.
$$

We shall use the following notation of spheres, balls, parabolic cylinders, and parabolic boundary
\begin{align*}
&S(x_0,r)=\{x\in\mathbb{R}^4\,|\,|x-x_0|=r\},\ \ S(r)=S(0,r),\ \ S=S(1);\\
&B(x_0,r)=\{x\in\mathbb{R}^4\,|\,|x-x_0|<r\},\ \ B(r)=B(0,r),\ \ B=B(1);\\
&Q(z_0,r)=B(x_0,r) \times (t_0-r^2,t_0),\ \ Q(r)=Q(0,r),\ \ Q=Q(1);\\
&\partial Q(z_0,r) = (S(x_0,r) \times [t_0-r^2,t_0))\cup \{(t_0-r^2,y)\,|\, y \in B(x_0,r)\},
\end{align*}
where $z_0=(x_0,t_0)$.

We also denote mean values of summable functions as follows:
\begin{align*}
[u]_{x_0,r}(t)&=\dfrac{1}{|B(r)|}\int_{B(x_0,r)}u(x,t)\,dx,\\
(u)_{z_0,r}&=\dfrac{1}{|Q(r)|}\int_{Q(z_0,r)}u\, dz.
\end{align*}
Here $|A|$ as usual denotes the Lebesgue measure of the set $A$.

Now, we introduce the following quantities:
\begin{align*}
A(r)&=A(r,z_0)=\esssup_{t_0-r^2\leq t\leq t_0}\dfrac{1}{r^2}\int_{B(x_0,r)}|u(x,t)|^2\,dx,\\
E(r)&=E(r,z_0)=\dfrac{1}{r^2}\int_{Q(z_0,r)}|\nabla u|^2\,dz,\\
C(r)&=C(r,z_0)=\dfrac{1}{r^3}\int_{Q(z_0,r)}|u|^3\,dz,\\
D(r)&=D(r,z_0)=\dfrac{1}{r^3}\int_{Q(z_0,r)}|p-[p]_{x_0,r}|^{3/2}\,dz,\\
F(r)&=F(r,z_0)=\dfrac{1}{r^2}\Big[\int_{t_0-r^2}^{t_0}
\big(\int_{B(x_0,r)}|p-[p]_{x_0,r}|^{1+\alpha}\,dx\big)^{\frac{1}{2\alpha}}\,dt\Big]^{\frac{2\alpha}{1+\alpha}},\\
G(r)&=G(r,z_0)=r^4\Big[\int_{Q(z_0,r)}|f|^6\,dz\Big]^{\frac{1}{3}}.
\end{align*}
where $\alpha \in (0,1)$ is a number to be specified later. Notice that all these quantities are invariant under the natural scaling:
$$
u_\lambda(x,t)=\lambda u(\lambda x,\lambda^2 t), \,\,
p_\lambda(x,t)=\lambda^2 p(\lambda x,\lambda^2 t),\,\,
f_\lambda(x,t)=\lambda^3 f(\lambda x,\lambda^2 t).
$$
We are going to estimate them in Section \ref{s2}. Note that the quantity $F$ is auxiliary and will only be used in the first two steps of the proof in order to give a weak decay estimate of other quantities.

We finish this short section by introducing a pressure decomposition which would play a important role in our proof. Let $\eta(x)$ be a smooth function on $\mathbb{R}^4$ supported in the unit ball $B(1)$, $0\leq \eta \leq 1$ and $\eta\equiv 1$ on $\bar{B}(2/3)$. Let $z_0$ be a given point in $Q_T$ and $r>0$ a real number such that $Q(z_0,r) \subset Q_T$. It's known that for a.e. $t \in (t_0-r^2,t_0)$, in the sense of distribution, one has
\begin{align*}
\Delta p &= -\dfrac{\partial^2}{\partial x_i \partial x_j}(u_iu_j)+\nabla \cdot f\\
&=-\dfrac{\partial^2}{\partial x_i \partial x_j}
\big((u_i-[u_i]_{x_0,r/2})(u_j-[u_j]_{x_0,r/2})\big)+\nabla \cdot f \ \ \text{in} \ \ B(x_0,r).
\end{align*}
This will hold for a weak solution to \eqref{ns}-\eqref{inc}. For these $t$, we consider the decomposition
\begin{equation}
\label{deco}
p=\tilde{p}_{x_0,r}+h_{x_0,r} \ \ in \ \ B(x_0,r),
\end{equation}
 where $\tilde{p}_{x_0,r}$ is the Newtonian potential of
\begin{equation*}
\dfrac{\partial^2}{\partial x_i \partial x_j}\big((u_i-[u_i]_{x_0,r/2})(u_j-[u_j]_{x_0,r/2})\eta(2(x-x_0)/r)\big)+\nabla \cdot \big(f\eta(2(x-x_0)/r)\big).
\end{equation*}
Then $h_{x_0,r}$ is harmonic in $B(x_0,r/3)$.

\mysection{The proof}\label{s2}
In our proof of the results, we will make use of the following well-known interpolation inequality.
\begin{lemma}
\label{interpolation}
For any function $u \in W_2^1(\mathbb{R}^4)$  and real numbers $q\in [2,4]$ and $r >0$,
\begin{multline*}
\int_{B_r}|u|^q \,dx \leq N(q)\Big[\big(\int_{B_r}|\nabla u|^2 \, dx\big)^{q-2}\big(\int_{B_r}|u|^2 \, dx\big)^{2-q/2}\\
+r^{-2(q-2)}\big(\int_{B_r}|u|^2\,dx\big)^{q/2}\Big].
\end{multline*}
\end{lemma}

Let $\cL:=\partial_t-\partial_{x_i}(a_{ij}\partial_{x_j})$ be a (possibly degenerate) divergence form parabolic operator with measurable coefficients which are bounded by a constant $K>0$. We will use the following Poincar\'e type inequality for solutions to parabolic equations. See, for instance, \cite[Lemma 3.1]{Krylov_05}.
\begin{lemma}
                    \label{lem11.31}
Let $z_0\in \bR^{d+1}$, $p\in (1,\infty)$, $r\in (0,\infty)$, $u\in \cH^1_{p,\text{loc}}(\bR^{d+1})$, $g=(g_1,\ldots,g_d),f\in L_{p,\text{loc}}(\bR^{d+1})$. Suppose that $u$ is a weak solution to $\cL u=\nabla\cdot g+f$ in $Q(z_0,r)$. Then we have
$$
\int_{Q(z_0,r)}|u(t,x)-(u)_{z_0,R}|^p\,dz\le Nr^p
\int_{Q(z_0,r)}\big(|\nabla u|^p+|g|^p+r^p|f|^p\big)\,dz,
$$
where $N=N(d,K,p)$.
\end{lemma}

Now we prove the main theorems in three steps.
\subsection{Step 1.}First, we control the quantities $A,C,D,F$ in a smaller ball by their values in a larger ball under the assumption that $E$ is sufficiently small. Here we follow the arguments in \cite{Dong_07}, which in turn used some ideas in \cite{Lady_99, Lin_98}.

\begin{lemma}
Suppose $\gamma \in (0,1), \rho >0$  are constants and $Q(z_0,\rho) \subset Q_T$. Then we have
\begin{equation}
\label{C_est}
C(\gamma \rho) \leq N[\gamma^{-3}A^{1/2}(\rho)E(\rho)+\gamma^{-9/2}A^{3/4}(\rho)E^{3/4}(\rho)+\gamma C(\rho)],
\end{equation}
where $N$ is a constant independent of $\gamma,\rho$ and $z_0$.
\end{lemma}
%
%
%
The proof can be found in \cite{Dong_07}.

\begin{lemma}
\label{D_est_lemma}
Suppose $\alpha \in [1/11,1/2], \gamma \in (0,1/8],\rho >0$ are constants and $Q(z_0,\rho) \subset Q_T$. Then for any $z_1 \in Q(z_0,\rho/8)$ we have
\begin{equation}
\label{F_est}
F(\gamma \rho,z_1) \leq N(\alpha)[\gamma^{-2}A^{\frac{1-\alpha}{1+\alpha}}(\dfrac{\rho}{2})E^{\frac{2\alpha}{1+\alpha}}(\dfrac{\rho}{2})+\gamma^{\frac{3-\alpha}{1+\alpha}}F(\rho)+\gamma^{-2} G^{1/2}(\rho)],
\end{equation}
where $N(\alpha)$ is a constant independent of $\gamma,\rho$ and $z_0$. In particular, for $\alpha=1/2$ we have
\begin{equation}
\label{D_est}
D(\gamma \rho,z_1) \leq N[\gamma^{-3}A^{1/2}(\dfrac{\rho}{2})E(\dfrac{\rho}{2})
+\gamma^{5/2}D(\rho)+\gamma^{-3}G^{3/4}(\rho)].
\end{equation}
Moreover, it holds that
\begin{equation}
\label{D_est_2}
D(\gamma \rho,z_1) \leq N(\alpha)[\gamma^{-3}(A(\dfrac{\rho}{2})+E(\dfrac{\rho}{2}))^{3/2}
+\gamma^{(9-3\alpha)/(2+2\alpha)}F^{3/2}(\rho)+\gamma^{-3}G^{3/4}(\rho)].
\end{equation}
\end{lemma}

\begin{proof}
First, we assume $1/3 \leq \alpha \leq 1/2$. Denote $r=\gamma \rho$. Recall the decomposition of $p$ introduced in \eqref{deco} and the definition of $\eta$. By using the Calder\'{o}n--Zygmund estimate, Lemma \ref{interpolation} with $q=2(1+\alpha)$, the Poincar\'{e} inequality, and the Sobolev embedding inequality, one has
\begin{align}
\nonumber&\int_{B(x_0,r)}|\tilde p_{x_0,r}|^{1+\alpha}\, dx \\
\nonumber&\leq N \int_{B(x_0,r/2)}|u-[u]_{x_0,r/2}|^{2(1+\alpha)}\,dx
+N\int_{B(x_0,r)}|\Delta^{-1}\nabla \cdot (f \eta(2(x-x_0)/r))|^{1+\alpha}\, dx \\
\nonumber&\leq N \big(\int_{B(x_0,r/2)}|\nabla u|^2\, dx\big)^{2\alpha}\big(\int_{B(x_0,r/2)}|u-[u]_{x_0,r/2}|^2\, dx\big)^{1-\alpha}\\
\nonumber&\,\,+N r^{-4\alpha}\big(\int_{B(x_0,r/2)}|u-[u]_{x_0,r/2}|^2\,dx\big)^{1+\alpha}
+N\big(\int_{B(x_0,r/2)}|f|^{\frac{4+4\alpha}{5+\alpha}}\,dx\big)^{\frac{5+\alpha}{4}}\\
&\leq N \big(\int_{B(x_0,r/2)}|\nabla u|^2\, dx\big)^{2\alpha} \big(\int_{B(x_0,r/2)}|u|^2\,dx\big)^{1-\alpha}
+N\big(\int_{B(x_0,r/2)}|f|^{\frac{4+4\alpha}{5+\alpha}}\,dx\big)^{\frac{5+\alpha}{4}},
\label{pf0}
\end{align}
where $\frac{4+4\alpha}{5+\alpha}\ge 1$.
Here we also used the obvious inequality
\begin{equation*}
\int_{B(x_0,r/2)}|u-[u]_{x_0,r/2}|^2\,dx \leq \int_{B(x_0,r/2)}|u|^2 \,dx.
\end{equation*}
Similarly,
\begin{align}
\nonumber \int_{B(x_0,\rho)}|\tilde p_{x_0,\rho}|^{1+\alpha}\, dx \leq & N \big(\int_{B(x_0,\rho/2)}|\nabla u|^2 \,dx\big)^{2\alpha} \big(\int_{B(x_0,\rho/2)}|u|^2\,dx\big)^{1-\alpha}\\&
+N\big(\int_{B(x_0,\rho/2)}|f|^{\frac{4+4\alpha}{5+\alpha}}\,dx\big)^{\frac{5+\alpha}{4}}.
\label{pf1}
\end{align}

Since $h_{x_0,\rho}$ is harmonic in $B(x_0,\rho/3)$, any Sobolev norm of $h_{x_0,\rho}$ in a smaller ball can be estimated by its $L_p$ norm in $B(x_0,\rho/3)$ for any $p\in [1,\infty]$. Thus, by using the Poincar\'{e} inequality one can obtain, for a.e. $t$,
\begin{align}
\nonumber&\int_{B(x_1,r)}|h_{x_0,\rho}-[h_{x_0,\rho}]_{x_1,r}|^{1+\alpha}\,dx\\
\nonumber&\leq N r^{1+\alpha}\int_{B(x_1,r)}|\nabla h_{x_0,\rho}|^{1+\alpha}\,dx\\
\nonumber&\leq N r^{5+\alpha} \sup_{B(x_1,r)}|\nabla h_{x_0,\rho}|^{1+\alpha}\\
\nonumber&\leq N (\dfrac{r}{\rho})^{5+\alpha} \int_{B(x_0,\rho/3)}|h_{x_0,\rho}(x,t)-[p]_{x_0,\rho}|^{1+\alpha}\,dx\\
&\leq N(\dfrac{r}{\rho})^{5+\alpha} [\int_{B(x_0,\rho)}|p(x,t)-[p]_{x_0,\rho}|^{1+\alpha}+
|\tilde{p}_{x_0,\rho}(x,t)|^{1+\alpha}\, dx].
\label{pf2}
\end{align}


Since $\tilde{p}_{x_0,r}+h_{x_0,r}=p=\tilde{p}_{x_0,\rho}+h_{x_0,\rho}$ in $B(x_1,r)$,
from \eqref{pf0}, \eqref{pf1} and \eqref{pf2} we get, for a.e. $t$,
\begin{align}
\nonumber &\int_{B(x_1,r)}|p(x,t)-[p]_{x_1,r}|^{1+\alpha}\,dx \\\nonumber &\leq \int_{B(x_1,r)}|\tilde{p}_{x_0,\rho}-[\tilde{p}_{x_0,\rho}]_{x_1,r}|^{1+\alpha}\,dx+\int_{B(x_1,r)}
|h_{x_0,\rho}-[h_{x_0,\rho}]_{x_1,r}|^{1+\alpha}\,dx
\\\nonumber &\leq \int_{B(x_1,r)}|\tilde{p}_{x_0,\rho}|^{1+\alpha}\,dx+\int_{B(x_1,r)}
|h_{x_0,\rho}-[h_{x_0,\rho}]_{x_1,r}|^{1+\alpha}\,dx\\
\nonumber &\leq N \big(\int_{B(x_0,\rho/2)}|\nabla u|^2\, dx\big)^{2\alpha} \big(\int_{B(x_0,\rho/2)}|u|^2\,dx\big)^{1-\alpha}\\
&\ \ +N\big(\dfrac{r}{\rho}\big)^{5+\alpha} \int_{B(x_0,\rho)}|p(x,t)-[p]_{x_0,\rho}|^{1+\alpha}\, dx+N\big(\int_{B(x_0,\rho)}|f|^{\frac{4+4\alpha}{5+\alpha}}\,dx\big)^{\frac{5+\alpha}{4}}.\label{pf5}
\end{align}
Raising to power $\dfrac{1}{2\alpha}$, integrating with respect to $t$ in $(t_1-r^2,t_1)$, and using H\"older's inequality complete the proof of \eqref{F_est} and \eqref{D_est}.

Now in the case $1/11 \leq \alpha < 1/3$, we cannot use the Sobolev embedding inequality directly in \eqref{pf0} because $\frac{4+4\alpha}{5+\alpha}< 1$. However, since function $\eta$ has compact support, by using H\"{o}lder inequality, we can get
\begin{align*}
&N \int_{B(x_0,r)}|\Delta^{-1}\nabla \cdot (f \eta(2(x-x_0)/r))|^{1+\alpha}\,dx \\&\leq N r^{4(1-\frac{1+\alpha}{1+\beta})} \big(\int_{B(x_0,r)}|\Delta^{-1}\nabla \cdot (f \eta(2(x-x_0)/r))|^{1+\beta}\,dx\big)^{\frac{1+\alpha}{1+\beta}} \\&\leq
Nr^{4(1-\frac{1+\alpha}{1+\beta})}
\big(\int_{B(x_0,r/2)}|f|^{\frac{4+4\beta}{5+\beta}}\,dx\big)^{\frac{5+\beta}{4}\frac{1+\alpha}{1+\beta}},
\end{align*}
where $1/3 \leq \beta \leq 1/2$. Noting that $\frac{4+4\beta}{5+\beta}< 6$ and $\frac{1+\alpha}{2\alpha}\leq 6$, we then prove \eqref{F_est} in the same way as the case $1/3 \leq \alpha \leq 1/2$.

To prove \eqref{D_est_2}, we use a slightly different estimate from \eqref{pf2}. Again, since $h$ is harmonic in $B(x_0,\rho/3)$, we have
\begin{align*}
\nonumber&\int_{B(x_1,r)}|h_{x_0,\rho}-[h_{x_0,\rho}]_{x_1,r}|^{3/2}\,dx\\
\nonumber&\leq N r^{3/2}\int_{B(x_1,r)}|\nabla h_{x_0,\rho}|^{3/2}\,dx\\
\nonumber&\leq N r^{11/2} \sup_{B(x_1,r)}|\nabla h_{x_0,\rho}|^{3/2}\\
\nonumber&\leq N \dfrac{r^{11/2}}{{\rho}^{3/2+6/(1+\alpha)}} [\int_{B(x_0,\rho/3)}|h_{x_0,\rho}(x,t)-[p]_{x_0,\rho}|^{1+\alpha}]^{\frac{3}{2(1+\alpha)}}\\
\nonumber&\leq N\dfrac{r^{11/2}}{{\rho}^{3/2+6/(1+\alpha)}} \bigg\{[\int_{B(x_0,\rho)}|p(x,t)-[p]_{x_0,\rho}|^{1+\alpha}\,dx]^{\frac{3}{2(1+\alpha)}}\\
&\,\,+[\int_{B(x_0,\rho)}|\tilde{p}_{x_0,\rho}(x,t)|^{1+\alpha} \,dx]^{\frac{3}{2(1+\alpha)}}\bigg\}.
\end{align*}
Similar to \eqref{pf5}, we obtain, for a.e. $t$,
\begin{align*}
\nonumber&\int_{B(x_1,r)}|p(x,t)-[p]_{x_1,r}|^{3/2}\,dx\\
\nonumber&\leq N \big(\int_{B(x_0,\rho/2)}|\nabla u|^2 \,dx\big) \big(\int_{B(x_0,\rho/2)}|u|^2\,dx\big)^{1/2}\\
\nonumber&\ \ +N\dfrac{r^{11/2}}{{\rho}^{3/2+6/(1+\alpha)}} \bigg \{\big[\int_{B(x_0,\rho)}|p(x,t)-[p]_{x_0,\rho}|^{1+\alpha}\,dx\big]^{\frac{3}{2(1+\alpha)}}\\\nonumber &\ \ +\big(\int_{B(x_0,\rho)}|\nabla u(x,t)|^2 \,dx\big)^{\frac{3\alpha}{1+\alpha}}\big(\int_{B(x_0,\rho)}|u(x,t)|^2 \,dx\big)^{\frac{3(1-\alpha)}{2(1+\alpha)}}\bigg\}\\&\ \ +N\big(\int_{B(x_0,\rho)}|f|^{\frac{12}{11}}\big)^{\frac{11}{8}}.
\end{align*}
Integrating with respect to $t$ in $(t_1-r^2,t_1)$ and applying H\"{o}lder inequality completes the proof of \eqref{D_est_2}.
\end{proof}
\begin{lemma}
Suppose $\theta \in (0,1/2], \rho >0$ are constants and $Q(z_0,\rho) \subset Q_T$. Then we have
\begin{equation*}
A(\theta \rho)+E(\theta \rho) \leq N \theta^{-2}\big[C^{2/3}(\rho)+C(\rho)+C^{1/3}(\rho)D^{2/3}(\rho)+G(\rho)\big].
\end{equation*}
In particular, when $\theta=1/2$ we have
\begin{equation}
\label{AE_est_2}
A(\rho/2)+E(\rho/2) \leq N \big[C^{2/3}(\rho)+C(\rho)+C^{1/3}(\rho)D^{2/3}(\rho)+G(\rho)\big].
\end{equation}
\end{lemma}

\begin{proof}
Let $r=\theta \rho$. In the energy inequality \eqref{energy}, we set $t=t_0$ and choose a suitable smooth cut-off function $\psi$ such that
\begin{align*}
\psi \equiv 0 \ \ \text{in} \ \ Q_{t_0}\backslash Q(z_0,\rho),\ \  0\leq \psi \leq 1 \ \ \text{in} \ \ Q_T,\\
\psi \equiv 1 \ \ \text{in} \ \ Q(z_0,r), \ \ |\nabla \psi| < N \rho^{-1}, \ \ |\partial_t \psi|+|\nabla^2 \psi| \leq N \rho^{-2}\ \ \text{in} \ \ Q_{t_0}.
\end{align*}
By using \eqref{energy} and because $u$ is divergence free, we get
\begin{multline*}
A(r)+2E(r) \leq \dfrac{N}{r^2}\Big[\dfrac{1}{\rho^2}\int_{Q(z_0,\rho)}|u|^2\,dz +\dfrac{1}{\rho}\int_{Q(z_0,\rho)}(|u|^2+2|p-[p]_{x_0,\rho}|)|u|\,dz\\
+\int_{Q(z_0,\rho)}|f||u|\,dz\Big].
\end{multline*}
Using H\"{o}lder's inequality and Young's inequality, one can obtain
\begin{equation*}
\int_{Q(z_0,\rho)}|u|^2\, dz \leq\big(\int_{Q(z_0,\rho)}|u|^3\, dz\big)^{2/3}\big(\int_{Q(z_0,\rho)}\,dz\big)^{1/3} \leq \rho^4C^{2/3}(\rho),
\end{equation*}
\begin{align*}
\int_{Q(z_0,\rho)}|p-[p]_{x_0,\rho}||u| \,dz
&\leq \big(\int_{Q(z_0,\rho)}|p-[p]_{x_0,\rho}|^{3/2}\,dz\big)^{2/3}
\big(\int_{Q(z_0,\rho)}|u|^3\,dz\big)^{1/3}\\
&\leq N \rho^3D^{2/3}(\rho)C^{1/3}(\rho),
\end{align*}
and
\begin{align*}
\int_{Q(z_0,\rho)}|f||u| \,dz &\leq \rho^2\int_{Q(z_0,\rho)}|f|^2\,dz + \dfrac{1}{\rho^2} \int_{Q(z_0,\rho)}|u|^2\,dz\\
&\leq \rho^6\Big[\int_{Q(z_0,\rho)}|f|^6\,dz\Big]^{1/3} + \dfrac{1}{\rho^2} \int_{Q(z_0,\rho)}|u|^2\,dz.
\end{align*}
Then the conclusion follows immediately.

\end{proof}
As a conclusion, we can obtain
\begin{proposition}
\label{prop1}
For any $\epsilon_0 >0$, there exists $\epsilon_1 >0$ small such that for any $z_0 \in Q_T $ satisfying
\begin{equation}
\limsup_{r \searrow 0} E(r)\leq \epsilon_1,
\label{e_cond}
\end{equation}
we can find $\rho_0$ sufficiently small such that
\begin{equation}
A(\rho_0)+E(\rho_0)+C(\rho_0)+D(\rho_0)+F(\rho_0) \leq \epsilon_0.
\label{e_conl}
\end{equation}

\end{proposition}

\begin{proof}
First, we prove \eqref{e_conl} without the presence of $F$ on the left-hand side. For a given point
\begin{equation*}
z_0=(x_0,t_0) \in Q_T
\end{equation*}
satisfying \eqref{e_cond}, choose $\rho_1 > 0$ such that $Q(z_0,\rho_1) \subset Q_T$. Then for any $\rho \in (0,\rho_1]$ and $\gamma \in (0,1/8)$, by using \eqref{AE_est_2} and Young's inequality,
\begin{equation*}
A(\gamma\rho)+E(\gamma \rho) \leq N\big[C^{2/3}(2\gamma\rho)+C(2\gamma\rho)+D(2\gamma\rho)+G(2\gamma\rho)\big].
\end{equation*}
Then, combining with \eqref{C_est} and \eqref{D_est} and using Young's inequality again, we have
\begin{align}
\nonumber&A(\gamma \rho)+E(\gamma \rho)+C(\gamma \rho)+D(\gamma \rho)\\
\nonumber&\leq N\big[\gamma^{2/3}C^{2/3}(\rho)+\gamma^{5/2}D(\rho)+\gamma C(\rho)+\gamma A(\rho)\big]\\
\nonumber&\ \ +N\gamma^{-100}\big(E(\rho)+E^3(\rho)+G(\rho)\big)+N \gamma^{2/3}\\
\nonumber&\leq N\gamma^{2/3}\big[A(\rho)+E(\rho)+C(\rho)+D(\rho)\big]+N \gamma^{2/3}\\
&\ \ +N\gamma^{-100}\big(E(\rho)+E^3(\rho)+G(\rho)\big).
\label{erttt}
\end{align}
Since $f\in L_{6,\text{loc}}(Q_T)$, we have
\begin{equation}
                            \label{11.30}
G(\rho) \leq \|f\|^2_{L_{6}(Q(z_0,\rho_1))}\rho^{4}.
\end{equation}
It is easy to see that for any $\epsilon_0 > 0$, there are sufficiently small real numbers $\gamma \leq 1/(2N)^{3/2}$ and $\epsilon_1$ such that if \eqref{e_cond} holds then for all small $\rho$ we have
\begin{equation*}
N\gamma^{2/3}+N\gamma^{-100}\big(E(\rho)+E^3(\rho)+G(\rho)\big)<\epsilon_0/2.
\end{equation*}
By using \eqref{erttt}, we can obtain
\begin{equation*}
A(\rho_0)+C(\rho_0)+D(\rho_0) \leq \epsilon_0
\end{equation*}
for some $\rho_0 >0$ small enough. To include $F$ in the estimate, it suffices to use \eqref{F_est}.
\end{proof}

\subsection{Step 2.}
In the second step, first we will estimate the values of $A$, $E$, $C$, and $F$ in a smaller ball by the values of themselves in a larger ball.
\begin{lemma}
\label{eeeeee}
Suppose $\rho >0$, $\theta \in (0,1/16]$ are constants and $Q(z_1,\rho) \subset Q_T$. Then we have
\begin{multline}
\label{eree}
A(\theta \rho)+E(\theta \rho) \leq N \theta^2A(\rho)+N\theta^{-3}\big[A(\rho)+E(\rho)+F(\rho)\big]^{3/2}\\
+N\theta^{-3}G^{3/4}(\rho)+N\theta^{-6}G(\rho),
\end{multline}
where $N$ is a constant independent of $\rho$, $\theta$, and $z_1$.
\end{lemma}
\begin{proof}
Let $r=\theta \rho$. Define the backward heat kernel as
\begin{equation*}
\Gamma(x,t)=\dfrac{1}{4\pi^2(r^2+t_1-t)^2}e^{-\frac{|x-x_1|^2}{2(r^2+t_1-t)}}.
\end{equation*}
In the energy inequality \eqref{energy} we put $t=t_1$ and choose $\psi=\Gamma\phi := \Gamma\phi_1(x)\phi_2(t)$, where $\phi_1$, $\phi_2$ are suitable smooth cut-off functions satisfying
\begin{align}
\nonumber\phi_1 \equiv 0 \ \ \text{in} \ \ \mathbb{R}^4\backslash B(x_1,\rho), \ \ 0\leq \phi_1\leq 1 \ \ \text{in} \ \ \mathbb{R}^4, \ \ \phi_1\equiv1 \ \ \text{in} \ \ B(x_1,\rho/2),\\
\nonumber \phi_2 \equiv 0 \ \ \text{in}\ \ (-\infty, t_1-\rho^2)\cup(t_1+\rho^2, +\infty), \ \ 0\leq\phi_2\leq 1 \ \ \text{in} \ \ \mathbb{R}, \\
\nonumber \phi_2 \equiv 1 \ \ \text{in} \ \ (t_1-\rho^2/4,t_1+\rho^2/4),\ \ |\partial_t\phi_2|\leq N \rho^{-2} \ \ \text{in} \ \ \mathbb{R},\\
|\nabla\phi_1| \leq N \rho^{-1}, |\nabla^2 \phi_1| \leq N \rho^{-2} \ \ \text{in} \ \ \mathbb{R}^4.
\label{qq}
\end{align}
By using the equality
\begin{equation*}
\Delta \Gamma + \Gamma_t =0,
\end{equation*}
we have
\begin{align}
\nonumber&\int_{B(x_1,\rho)}|u(x,t)|^2\Gamma(t,x)\phi(x,t) \,dx +2 \int_{Q(z_1,\rho)}|\nabla u|^2 \Gamma \phi \,dz\\
\nonumber&\leq \int_{Q(z_1,\rho)}\big\{|u|^2(\Gamma \phi_t+\Gamma \Delta \phi+2\nabla\phi\nabla\Gamma)+(|u|^2+2p)u\cdot (\Gamma\nabla \phi+\phi\nabla \Gamma)\big\}\,dz\\
&\ \ +\int_{Q(z_1,\rho)}|f||u||\Gamma\phi|\,dz.
\label{ggg}
\end{align}

With straightforward computations, it is easy to see the following three properties:
(i) For some constant $c>0$, on $\bar{Q}(z_1,r)$ it holds that
\begin{equation*}
\Gamma \phi = \Gamma \geq c r^{-4}.
\end{equation*}
(ii) For any $z \in Q(z_1,\rho)$, we have
\begin{equation*}
|\Gamma(z)\phi(z)| \leq N r^{-4},\ \ |\phi(z)\nabla\Gamma(z)|+|\nabla\phi(z)\Gamma(z)| \leq N r^{-5}.
\end{equation*}
(iii) For any $z \in Q(z_1,\rho)\backslash Q(z_1,r)$, we have
\begin{equation*}
|\Gamma(z)\phi_t(z)|+|\Gamma(z)\Delta\phi(z)|+|\nabla\phi\nabla\Gamma| \leq N\rho^{-6}.
\end{equation*}
These properties together with \eqref{qq} and \eqref{ggg} yield
\begin{equation}
A(r)+E(r) \leq N\big[\theta^2A(\rho)+\theta^{-3}(C(\rho)+D(\rho))+\theta^{-6}G(\rho)\big].
\label{hh}
\end{equation}
Owing to \eqref{interpolation} with $q=3$, we can get
\begin{equation}
C(\rho/8)\leq NC(\rho)\leq N\big[A(\rho)+E(\rho)\big]^{3/2}.
\label{hhh}
\end{equation}
By using \eqref{D_est_2} with $\gamma=1/8$, we have
\begin{equation}
D(\rho/8)\leq N\big[A(\rho)+E(\rho)+F(\rho)\big]^{3/2}+N G^{3/4}(\rho).
\label{hhhh}
\end{equation}
Upon combining \eqref{hh} (with $\rho/8$ in place of $\rho$) to \eqref{hhhh} together, the lemma is proved.
\end{proof}
\begin{lemma}
\label{9i}
Suppose $\rho>0$ is constant and $Q(z_1,\rho) \subset Q_T$. Then we can find $\theta_1 \in (0,1/256]$ small such that
\begin{multline}
A(\theta_1\rho)+E(\theta_1\rho)+F(\theta_1\rho) \leq \dfrac{1}{2}\big[A(\rho)+E(\rho)+F(\rho)\big]\\
+N(\theta_1)\big[A(\rho)+E(\rho)+F(\rho)\big]^{3/2}
+N(\theta_1)\big[G(\rho)+G^{1/2}(\rho)\big],
\label{90}
\end{multline}

where $N$ is a constant independent of $\rho$ and $z_1$.
\end{lemma}
\begin{proof}
Due to \eqref{F_est} and \eqref{eree}, for any $\gamma,\theta \in (0,1/16]$, we have
\begin{align}
\nonumber F(\gamma \theta\rho)\leq &N\big[\gamma^{-2}(A(\theta\rho)+E(\theta\rho))
+\gamma^{(3-\alpha)/(1+\alpha)}F(\theta\rho)
+\gamma^{-2}G^{1/2}(\theta\rho)\big]\\
\nonumber \leq &N\gamma^{-2}\theta^2A(\rho)+\gamma^{(3-\alpha)/(1+\alpha)}\theta^{-2}F(\rho)
+\gamma^{-2}G^{1/2}(\rho)\\
&+N\gamma^{-2}\theta^{-3}\big[A(\rho)+E(\rho)+F(\rho)\big]^{3/2},\label{ii}
\end{align}
and
\begin{align}
\nonumber A(\gamma\theta\rho)+E(\gamma\theta\rho)&\leq N(\gamma\theta)^2A(\rho)+N(\gamma\theta)^{-3}\big[A(\rho)+E(\rho)+F(\rho)\big]^{3/2}\\
&\,\,+N(\gamma\theta)^{-6}G(\rho)
+N(\gamma\theta)^{-3}G^{3/4}(\rho).\label{iii}
\end{align}
Now we set $\alpha=1/5$ such that $(3-\alpha)/(1+\alpha)=7/3 >2$, and choose and fix $\gamma$ and $\theta$ sufficiently small such that
\begin{equation*}
N\big[\gamma^{-2}\theta^2+\gamma^{7/3}\theta^{-2}+\gamma^2\theta^2\big] \leq \dfrac{1}{2}.
\end{equation*}
Upon adding \eqref{ii} and \eqref{iii}, we can obtain
\begin{align*}
A(\gamma\theta\rho)+E(\gamma\theta\rho)+F(\gamma\theta\rho) &\leq \dfrac{1}{2}A(\rho)+N\big[A(\rho)+E(\rho)+F(\rho)\big]^{3/2}\\
&\,\,+N\big[G(\rho)+G^{1/2}(\rho)+G^{3/4}(\rho)\big],
\end{align*}
where $N$ only depends on $\theta$ and $\gamma$. After putting $\theta_1=\gamma\theta$, the lemma is proved.
\end{proof}

In the next proposition we will study the decay property of $A$, $C$,$E$ and $F$ as the radius $\rho$ goes to zero.
\begin{proposition}
There exists $\epsilon_0>0$ satisfying the following property. Suppose that for some $z_0 \in Q_T $ and $\rho_0 >0$ satisfying $Q(z_0,\rho_0) \subset Q_T$ we have
\begin{equation}
C(\rho_0)+D(\rho_0)+F(\rho_0) +G(\rho_0)\leq \epsilon_0.
\label{cond_cdf}
\end{equation}
Then we can find $N>0$ and $\alpha_0 \in (0,1)$ such that for any $\rho \in (0,\rho_0/8)$ and $z_1 \in Q(z_0,\rho/8)$, the following inequality will hold uniformly
\begin{equation}
A(\rho,z_1)+C(\rho,z_1)+E(\rho,z_1)+F(\rho,z_1) + D(\rho,z_1)\leq N \rho ^{\alpha_0},
\label{cdf}
\end{equation}
where $N$ is a positive constant independent of $\rho$ and $z_1$.
\end{proposition}
\begin{proof}
Fix the constant $\theta_1 \in (0,1/256]$ from Lemma \ref{9i} and let $N(\theta_1)>0$ is the same constant from \eqref{90}.  Due to \eqref{AE_est_2}, \eqref{cond_cdf}, and \eqref{F_est}, we first choose $\epsilon'>0$ and then $\epsilon_0=\epsilon_0(\epsilon')>0$ sufficiently small such that,
\begin{equation}
N(\theta_1)\sqrt{\epsilon'}\leq 1/4,\ \ N(\theta_1)\big(\epsilon_0+\epsilon_0^{1/2}\big) \leq \epsilon'/2,
\label{pop}
\end{equation}
\begin{equation*}
A(\rho_0/2)+E(\rho_0/2)\leq \dfrac{\epsilon'}{32},
\end{equation*}
and for any $z_1 \in Q(z_0,\rho_0/8)$,
\begin{equation*}
F(\rho_0/8,z_1) \leq N\big[A^{2/3}(\dfrac{\rho_0}{2})
E^{1/3}(\dfrac{\rho_0}{2})+F(\rho_0)
+G^{1/2}(\rho_0)\big]\leq \dfrac{\epsilon'}{32}.
\end{equation*}
By using
\begin{equation*}
Q(z_1,\rho_0/8)\subset Q(z_0,\rho_0/2)\subset Q_T,
\end{equation*}
we then have
\begin{equation*}
\phi(\rho_0):=A(\rho_0/8,z_1)+E(\rho_0/8,z_1)+F(\rho_0/8,z_1) \leq \epsilon'.
\end{equation*}
By using \eqref{90} and \eqref{pop} with $\rho=\rho_0/8$, we obtain inductively that
\begin{equation*}
\phi(\theta_1^k\rho_0)=A(\theta_1^k\rho_0/8,z_1)+E(\theta_1^k\rho_0/8,z_1)+F(\theta_1^k\rho_0/8,z_1) \leq \epsilon'
\end{equation*}
(holding for $k=1,2,...$). It then similarly follows from \eqref{90} and \eqref{pop} that
\begin{equation}
\phi(\theta_1^k\rho_0)\leq \dfrac{3}{4}\phi(\theta_1^{k-1}\rho_0)+N_1(\theta_1^{k-1}\rho_0)^2,
\label{it}
\end{equation}
where we have used the estimate
\begin{equation*}
G(\theta_1^{k-1}\rho_0/8,z_1)+G^{1/2}(\theta_1^{k-1}\rho_0/8,z_1)
 \leq N(\|f\|_{L_6(Q(z_0,\rho_0/2)})(\theta_1^{k-1}\rho_0)^2.
\end{equation*}
Now we use a standard iteration argument to obtain the decay rate of $\phi$. We iterate \eqref{it} to obtain
\begin{equation*}
\phi(\theta_1^k\rho_0)\leq (\dfrac{3}{4})^k \Big[\phi(\rho_0)+\dfrac{2N_1}{1-\theta_1}\rho_0^2\Big],
\end{equation*}
where we have used that $\theta_1 <3/4$.
Since $\rho \in (0,\rho_0/32)$ we can find $k$ such that $\theta_1^k\dfrac{\rho_0}{8} < 4\rho \leq \theta_1^{k-1}\dfrac{\rho_0}{8}$. Then
\begin{align*}
A(\rho,z_1)+E(\rho,z_1)+F(\rho,z_1) &\leq N(\theta_1)(\dfrac{3}{4})^k \Big[\phi(\rho_0)+\dfrac{2N_1}{1-\theta_1}\rho_0^2\Big] +N(\theta_1)\rho^2 \\ &\leq N\rho^{\alpha_0},
\end{align*}
where $N=N(\theta_1,\phi(\rho_0),N_1,\rho_0)$ and $\alpha_0=\dfrac{\log(3/4)}{\log(\theta_1)} >0 $. This yields \eqref{cdf} for the terms $A$, $E$, and $F$. The inequality for $C(\rho,z_1)$ follows from \eqref{hhh} and the inequality for $D(\rho,z_1)$ follows by \eqref{D_est_2}.
\end{proof}
\subsection{Step 3}
In the final step, we are going to use a bootstrap argument to successively improve the decay estimate \eqref{cdf}. However, as we will show below, the bootstrap argument itself only gives the decay of $E(\rho)$ no more than $\rho^{5/3}$, for instance, one can get an estimate like
\begin{equation*}
\int_{Q(z_1,\rho)}|\nabla u|^2\,dz \leq N\rho^{2+\frac{5}{3}}
\end{equation*}
for any $\rho$ sufficiently small. Unfortunately, this decay estimate is not enough for the H\"{o}lder regularity of $u$ since the spatial dimension is four (so that we need the decay exponent $4+\epsilon$ according to Morrey's lemma). We shall use parabolic regularity to fill in this gap.

First we prove Theorem \ref{th2}. We begin with the bootstrap argument. We will choose an increasing sequence of real numbers $\{\alpha_k\}_{k=1}^{m}\in (\alpha_0,5/3]$.

Under the condition \eqref{cdf}, we claim that the following estimate hold uniformly for all $\rho >0$ sufficiently small and $z_1 \in Q(z_0,\rho_0/8)$ over the range of $\{\alpha_k\}_{k=0}^m$:
\begin{align}
\label{esti}
A(\rho,z_1)+E(\rho,z_1)\leq N \rho^{\alpha_k},\,\, C(\rho,z_1)+D(\rho,z_1)\leq N\rho^{3\alpha_k/2}.
\end{align}
We prove this via iteration. The $k=0$ case for \eqref{esti} was proved in \eqref{cdf} with a possibly different exponent $\alpha_0$.

Now suppose \eqref{cdf} holds with the exponent $\alpha_k$. We first estimate $A(\rho,z_1)$ and $E(\rho,z_1)$. Let $\rho=\tilde{\theta}\tilde{\rho}$ where $\tilde{\theta}=\rho^{\mu}$, $\tilde{\rho}=\rho^{1-\mu}$ and $\mu \in (0,1)$ to be determined. We use \eqref{hh}, \eqref{esti} (for $\alpha_k$), and \eqref{11.30} to obtain
\begin{equation*}
A(\rho)+E(\rho) \leq N \rho^{2\mu}\rho^{\alpha_k(1-\mu)}+N \rho^{-3\mu}\rho^{\frac{3}{2}\alpha_k(1-\mu)}+N \rho^{4(1-\mu)}\rho^{-6\mu}.
\end{equation*}
Choose $\mu=\dfrac{\alpha_k}{10+\alpha_k}$, then \eqref{esti} is proved for $A(\rho)+E(\rho)$ with the exponent of
\begin{align*}
\tilde{\alpha}_{k+1}&:=\min\Big\{2\mu+\alpha_k(1-\mu),
\frac{3}{2}\alpha_k(1-\mu)-3\mu,4(1-\mu)-6\mu\Big\}\\
&=\dfrac{12}{10+\alpha_k}\alpha_k \in (\alpha_k,2).
\end{align*}
Then the estimate in \eqref{esti} (with $\tilde{\alpha}_{k+1}$) for $C(\rho,z_1)$ follows from \eqref{hhh}.
To prove the estimate in \eqref{esti} (with $\tilde{\alpha}_{k+1}$) for  $D(\rho,z_1)$ we will use Lemma \ref{D_est_lemma}. From \eqref{D_est} and \eqref{11.30}, we have
\begin{equation*}
D(\gamma \rho,z_1) \leq N\big[\gamma^{-3}\rho^{\frac{3}{2}\tilde{\alpha}_{k+1}}+\gamma^{5/2}D(\rho,z_1)+\gamma^{-3}\rho^3\big].
\end{equation*}
For any $r$ small, we take the supremum on both sides with respect to $\rho \in (0,r)$ and get
\begin{align*}
\sup_{\rho\in(0,r]}D(\gamma \rho,z_1) &\leq N\gamma^{-3}r^{\frac{3}{2}\tilde{\alpha}_{k+1}}+N\gamma^{5/2}\sup_{\rho\in(0,r]}D(\rho,z_1)+\gamma^{-3}r^3.
\end{align*}
Now set
\begin{equation}\label{d}
\alpha_{k+1}= \min\big\{\frac{5}{3}, \tilde{\alpha}_{k+1}\big\}.
\end{equation} 
By using a well-known iteration argument, similar to \eqref{it}, we obtain the estimate in \eqref{esti} (with $\alpha_{k+1}$) for $D(\rho)$. And the estimates in \eqref{esti} (with $\alpha_{k+1}$) for $A(\rho)+E(\rho)$ and $C(\rho)$ are still true. Then we have shown how to build the increasing sequence of $\{\alpha_k\}$ for which \eqref{esti} holds.

Moreover,
\begin{equation*}
2-\tilde\alpha_{k+1}=\dfrac{10}{10+\tilde{\alpha}_k}(2-\tilde{\alpha}_k) \leq \dfrac{10}{10+\alpha_0}(2-\tilde{\alpha}_k).
\end{equation*}
Thus, we can find a $m$ that $\alpha_m=\dfrac{5}{3}$ according to \eqref{d} because otherwise $\alpha_k=\tilde{\alpha}_k \rightarrow 2$ as $k \rightarrow \infty$.

 We have got the following estimates via the bootstrap argument:
\begin{equation}
\sup_{t_1-\rho^2 \leq t\leq t_1}\int_{B(x_1,\rho)}|u(x,t)|^2\,dx \leq N \rho^{2+\frac{5}{3}},
\label{pp}
\end{equation}
\begin{equation}
\int_{Q(z_1,\rho)}|u|^3+|p-[p]_{x_1,\rho}|^{3/2} \,dz \leq N \rho^{3+\frac{5}{2}}.
\label{ppp}
\end{equation}

Now we rewrite \eqref{ns} (in the weak sense) into
\begin{equation}
                                        \label{eq11.55}
\partial_t u_i- \Delta u_i=-\partial_j(u_iu_j)-\partial_i p+f_i.
\end{equation}

Finally, we use the parabolic regularity theory to improve the decay estimate of mean oscillations of $u$ and then complete the proof.
Due to \eqref{pp} and \eqref{ppp} , there exists $\rho_1 \in (\rho/2,\rho)$ such that
\begin{equation}
\int_{B(x_1,\rho_1)}|u(x,t_1-\rho_1^2)|^2\,dx \leq N \rho^{2+\frac{5}{3}},\,\,
\quad\int_{t_1-\rho_1^2}^{t_1}\int_{S(x_1,\rho_1)}|u|^3 \,dx\,dt \leq N \rho^{2+\frac{5}{2}}.
\label{oo}
\end{equation}
Let $v$ be the unique weak solution to the heat equation
\begin{equation*}
\partial_t v-\Delta v=0 \ \ \text{in} \ \ Q(z_1,\rho_1)
\end{equation*}
with the boundary condition $v_i=u_i$ on $\partial Q(z_1,\rho_1)$. It follows from the standard estimates for heat equation, H\"{o}lder's inequality, and \eqref{oo} that
\begin{align}
 &\sup_{Q(z_1,\rho_1/2)}|\nabla v| \nonumber\\
 &\leq N \rho_1^{-6}\int_{t_1-\rho_1^2}^{t_1}\int_{S(x_1,\rho_1)}|v|\,  dx\,dt+N\rho_1^{-5}\int_{B(x_1,\rho_1)}|v(x,t_1-\rho_1^2)|\,dx\nonumber\\
 &\leq N \rho^{-2+\frac{5}{6}}.
\label{uu}
\end{align}
Denote $w=u-v.$ Then $w$ satisfies the linear parabolic equation
\begin{equation*}
\partial_t w_i -\Delta w_i = -\partial_j(u_iu_j)-\partial_i (p-[p]_{x_1,\rho})+f_i
\end{equation*}
with zero boundary condition. By the classical $L_p$ estimate for  parabolic equations, we have
\begin{align*}
\nonumber\|\nabla w\|_{L_{3/2}(Q(z_1,\rho_1))} &\leq N \||u|^2\|_{L_{3/2}(Q(z_1,\rho_1))} +N \|p-[p]_{x_1,\rho}\|_{L_{3/2}(Q(z_1,\rho_1))}\\&\,\,
+N\rho_1\|f\|_{L_{3/2}(Q(z_1,\rho_1))},
\end{align*}
which together with \eqref{ppp} and the condition $f\in L_{6,\text{loc}}$ yields
\begin{equation}
                        \label{eq10.42}
\int_{Q(z_1,\rho_1)}|\nabla w|^{3/2}\,dz\le N\rho^{11/2}.
\end{equation}
Since $|\nabla u|\leq |\nabla w|+|\nabla v|$, we combine \eqref{uu} and \eqref{eq10.42} to obtain, for any $r \in (0,\rho/4)$, that
\begin{equation*}
\int_{Q(z_1,r)}|\nabla u|^{3/2} \,dz \leq N \rho^{11/2}+ r^{6}\rho^{-3+\frac{5}{4}}.
\end{equation*}
Upon taking $r=\rho^{29/24}/4$ (with $\rho$ small), we deduce
\begin{equation}
                                \label{eq11.50}
\int_{Q(z_1,r)}|\nabla u|^{3/2}\,dz \leq N r^{\beta},
\end{equation}
where
\begin{equation*}\beta=\frac{132}{29}>6-\frac 3 2.
\end{equation*}
Since $u\in \cH^1_{3/2,\text{loc}}$ is a weak solution to \eqref{eq11.55}, it then follows from Lemma \ref{lem11.31}, \eqref{eq11.50}, and \eqref{ppp} with $r$ in place of $\rho$ that
\begin{align*}
\nonumber &\int_{Q(z_1,r)}|u-(u)_{z_1,r}|^{3/2}\,dz \\
\nonumber &\leq N r^{3/2}\int_{Q(z_1,r)}\big|\nabla u|^{3/2}+(|u|^2)^{3/2}+|p-[p]_{x_1,r}|^{3/2}+r^{3/2}|f|^{3/2}\big)\,dz \\
\nonumber &\leq N r^{\beta+3/2}.
\end{align*}
By Campanato's characterization of H\"older
continuous functions (see, for instance, \cite[Lemma 4.3]{Lieb_96}), $u$ is H\"{o}lder continuous in a neighborhood of $z_0$. This completes the proof of Theorem \ref{th2}.

Theorem \ref{mainthm} then follows from Theorem \ref{th2} by applying Proposition \ref{prop1}. Finally, Theorem \ref{th3} is deduced from Theorem \ref{mainthm} by using standard argument in the geometric measure theory, which is explained, for example, in \cite{CKN_82}.

\section*{Acknowledgements.} H. Dong was partially supported by the NSF under agreement DMS-1056737. X. Gu was sponsored by the China Scholarship Council for one year study at Brown University and was partially supported by the NSFC (grant No. 11171072) and the Innovation  Program  of  Shanghai  Municipal  Education  Commission  (grant
No. 12ZZ012).

\end{document}